\documentclass[10pt,english,reqno]{amsart}

\usepackage{amsmath,amsthm,amssymb,amsfonts}
\usepackage[colorlinks=true,linkcolor=blue,urlcolor=blue,citecolor=blue]{hyperref}
\usepackage{mathtools,mathdots,bm}
\usepackage{booktabs,float,caption,enumitem}
\usepackage{setspace,diagbox}

\theoremstyle{plain}
\newtheorem{theorem}{Theorem}
\newtheorem{lemma}[theorem]{Lemma}

\numberwithin{equation}{section}
\numberwithin{theorem}{section}
\numberwithin{table}{section}
\numberwithin{figure}{section}

\setlist[enumerate]{label=(\roman*),font=\rm,leftmargin=1.2cm,itemsep=1pt,
parsep=1pt,before={\parskip=1pt}}

\DeclarePairedDelimiter{\floor}{\lfloor}{\rfloor}

\DeclarePairedDelimiter{\res}{\langle}{\rangle}
\DeclareMathOperator{\denom}{denom}

\newcommand{\set}[1]{\left\{#1\right\}}
\newcommand{\phrase}[1]{``\textsl{#1}"}
\newcommand{\andq}{\quad \text{and} \quad}
\newcommand{\pmods}[1]{(\mathrm{mod}\;#1)}
\newcommand{\iffs}{\Leftrightarrow}
\newcommand{\iffq}{\;\iff\;}

\newcommand{\mids}{\,\mid\,}
\newcommand{\nmids}{\,\nmid\,}
\newcommand{\valueat}[1]{{\,}_{\big|\, #1}}

\newcommand{\NN}{\mathbb{N}}
\newcommand{\ZZ}{\mathbb{Z}}
\newcommand{\QQ}{\mathbb{Q}}

\newcommand{\BN}{\mathbf{B}}
\newcommand{\BS}{\mathcal{B}}
\newcommand{\BF}{\mathfrak{B}}
\newcommand{\DN}{\mathbf{D}}
\newcommand{\DC}{\mathcal{D}}
\newcommand{\IM}{\mathcal{I}}

\newcommand{\linktext}[1]{\scriptsize\texttt{#1}}
\newcommand{\arXiv}[1]{\href{https://arxiv.org/abs/#1}{\linktext{arXiv:#1~[math.NT]}}}
\newcommand{\badw}[1]{\href{http://publikationen.badw.de/de/#1}{\linktext{BAdW:\,#1}}}

\title[Shifted sums of the Bernoulli numbers]
{Shifted sums of the Bernoulli numbers,\\reciprocity, and denominators}
\author{Bernd C. Kellner}
\subjclass[2020]{11B68 (Primary), 11B65 (Secondary)}
\address{G\"ottingen, Germany}
\email{bk@bernoulli.org}
\keywords{Bernoulli number and polynomial, shifted sum, reciprocity relation,
denominator}

\begin{document}

\begin{abstract}
We consider the numbers $\mathcal{B}_{r,s} = (\mathbf{B}+1)^r \, \mathbf{B}^s$
(in umbral notation $\mathbf{B}^n = \mathbf{B}_n$ with the Bernoulli numbers)
that have a well-known reciprocity relation, which is frequently found in the
literature and goes back to the 19th century. In a recent paper,
self-reciprocal Bernoulli polynomials, whose coefficients are related to these
numbers, appeared in the context of power sums and the so-called Faulhaber
polynomials. The numbers $\mathcal{B}_{r,s}$ can be recursively expressed by
iterated sums and differences, so it is not obvious that these numbers do not
vanish in general. As a main result among other properties, we show the
non-vanishing of these numbers, apart from exceptional cases. We further derive
an explicit product formula for their denominators, which follows from a
von Staudt--Clausen type relation.
\end{abstract}

\maketitle


\section{Introduction}

Define the numbers
\begin{equation} \label{eq:bs-def}
  \BS_{r,s} = \sum_{\nu=0}^{r} \binom{r}{\nu} \BN_{s+\nu}
\end{equation}
with rank $r \geq 0$ and shift $s \geq 0$, where the Bernoulli numbers $\BN_n$
are defined by
\[
  \frac{t}{e^t-1} = \sum_{n \geq 0} \BN_n \frac{t^n}{n!}.
\]
The numbers $\BN_n$ are rational, whereas $\BN_n = 0$ for odd $n \geq 3$, and
the sign of $\BN_n$ alternates for even $n \geq 2$. The numbers $\BS_{r,s}$
satisfy the well-known reciprocity relation
\begin{equation} \label{eq:bs-recip}
  (-1)^r \BS_{r,s} = (-1)^s \BS_{s,r},
\end{equation}
see Section~\ref{sec:umbral} for some historical background.
By $\binom{r+1}{\nu} = \binom{r}{\nu} + \binom{r}{\nu-1}$ for $\nu \geq 1$,
the recurrence of the binomial coefficients gives rise to
\begin{align}
  \BS_{r+1,s} &= \BS_{r,s} + \BS_{r,s+1}. \label{eq:bs-recur} \\
\shortintertext{It easily follows by induction the general rule for $n \geq 0$ that}
  \BS_{r+n,s} &= \sum_{\nu=0}^{n} \binom{n}{\nu} \BS_{r,s+\nu}. \label{eq:bs-recur-2}
\end{align}

For $r,s \geq 0$, we have the special values
\begin{equation} \label{eq:bs-bn}
\begin{alignedat}{2}
  \BS_{0,s} &= \BN_s, &\quad \BS_{1,s} &= \BN_s + \BN_{s+1}, \\
  \BS_{r,0} &= (-1)^r \, \BN_r, &\quad \BS_{r,1} &= (-1)^{r+1}(\BN_r + \BN_{r+1}).
\end{alignedat}
\end{equation}
Note that $\BS_{1,s} = (-1)^{s+1} \BS_{s,1} \neq 0$, since only one of the
numbers $\BN_s$ and $\BN_{s+1}$ is zero when $s \geq 2$. For $s \in \set{0,1}$,
see the computed values of $\BS_{1,s}$ in the following table. In particular,
the first row contains the Bernoulli numbers.

\begin{table}[H] \small
\newcommand{\mc}[1]{\multicolumn{1}{r}{\hspace*{1.5em}$#1$}}
\setstretch{1.3}
\begin{center}
\begin{tabular}{c|*{9}{r}}
  \toprule
  \diagbox[width=2.2em, height=2.2em]{$r$}{$s$} &
    \mc{0} & \mc{1} & \mc{2} & \mc{3} & \mc{4} & \mc{5} & \mc{6} & \mc{7} & \mc{8} \\\hline
  $0$ & $1$ & $-\frac{1}{2}$ & $\frac{1}{6}$ & $0$ & $-\frac{1}{30}$ & $0$ & $\frac{1}{42}$ & $0$ & $-\frac{1}{30}$ \\
  $1$ & $\frac{1}{2}$ & $-\frac{1}{3}$ & $\frac{1}{6}$ & $-\frac{1}{30}$ & $-\frac{1}{30}$
      & $\frac{1}{42}$ & $\frac{1}{42}$ & $-\frac{1}{30}$ & $-\frac{1}{30}$ \\
  $2$ & $\frac{1}{6}$ & $-\frac{1}{6}$ & $\frac{2}{15}$ & $-\frac{1}{15}$ & $-\frac{1}{105}$
      & $\frac{1}{21}$ & $-\frac{1}{105}$ & $-\frac{1}{15}$ & $\frac{7}{165}$ \\
  $3$ & $0$ & $-\frac{1}{30}$ & $\frac{1}{15}$ & $-\frac{8}{105}$ & $\frac{4}{105}$
      & $\frac{4}{105}$ & $-\frac{8}{105}$ & $-\frac{4}{165}$ & $\frac{32}{165}$ \\
  $4$ & $-\frac{1}{30}$ & $\frac{1}{30}$ & $-\frac{1}{105}$ & $-\frac{4}{105}$ & $\frac{8}{105}$
      & $-\frac{4}{105}$ & $-\frac{116}{1155}$ & $\frac{28}{165}$ & $\frac{2524}{15015}$ \\
  $5$ & $0$ & $\frac{1}{42}$ & $-\frac{1}{21}$ & $\frac{4}{105}$ & $\frac{4}{105}$
      & $-\frac{32}{231}$ & $\frac{16}{231}$ & $\frac{5072}{15015}$ & $-\frac{8128}{15015}$ \\
  $6$ & $\frac{1}{42}$ & $-\frac{1}{42}$ & $-\frac{1}{105}$ & $\frac{8}{105}$ & $-\frac{116}{1155}$
      & $-\frac{16}{231}$ & $\frac{6112}{15015}$ & $-\frac{3056}{15015}$ & $-\frac{22928}{15015}$ \\
  $7$ & $0$ & $-\frac{1}{30}$ & $\frac{1}{15}$ & $-\frac{4}{165}$ & $-\frac{28}{165}$ & $\frac{5072}{15015}$
      & $\frac{3056}{15015}$ & $-\frac{3712}{2145}$ & $\frac{1856}{2145}$ \\
  $8$ & $-\frac{1}{30}$ & $\frac{1}{30}$ & $\frac{7}{165}$ & $-\frac{32}{165}$ & $\frac{2524}{15015}$
      & $\frac{8128}{15015}$ & $-\frac{22928}{15015}$ & $-\frac{1856}{2145}$ & $\frac{362624}{36465}$ \\
  \bottomrule
\end{tabular}

\caption{First numbers $\BS_{r,s}$.}
\label{tbl:bs-coeff}
\end{center}
\end{table}

The shifted sum \eqref{eq:bs-def} can be interpreted as a binomial transform,
but it is rather an iterated difference of the Bernoulli numbers. Using
$\Delta$ as the forward difference operator (see the next section for notation)
and the symmetry by \eqref{eq:bs-recip}, we obtain
\[
  \BS_{r,s} = (-1)^{r+s} \Delta^{\!r} \, (-1)^\nu \, \BN_\nu \valueat{\nu = s}
    = \Delta^{\!s} \, (-1)^\nu \, \BN_\nu \valueat{\nu = r}.
\]
Note that $(-1)^n \, \BN_n = \BN_n$ except for $n=1$.
The analog to \eqref{eq:bs-recur-2} is given by
\[
  \BS_{r+n,s} = (-1)^{s+n} \Delta^{\!n} \, (-1)^{\nu} \, \BS_{r,\nu} \valueat{\nu = s}
  \andq \BS_{r,s+n} = \Delta^{\!n} \, \BS_{\nu,s} \valueat{\nu = r}.
\]

Together with \eqref{eq:bs-recur-2}, the numbers $\BS_{r+n,s}$ are given by
iterated sums over the shifts (respectively, columns), while the numbers
$\BS_{r,s+n}$ are given by iterated differences over the ranks (respectively,
rows of Table~\ref{tbl:bs-coeff}).

In 1877, Seidel~\cite{Seidel:1877} already computed a difference scheme of the
Bernoulli numbers in triangular form, moving from column to column.
The resulting table is reprinted as Figure~\ref{fig:seidel} in the appendix,
which coincides with Table~\ref{tbl:bs-coeff}.

Rewriting Table~\ref{tbl:bs-coeff} by its anti-diagonals, the symmetry of
$\BS_{r,s}$ is then visible in each line. More precisely, the sequence of those
elements when viewed as absolute values is palindromic. Furthermore, the sum
over these anti-diagonal elements vanishes for $n \geq 1$ as follows.

\begin{lemma} \label{lem:diag-sum}
For $n \geq 1$, we have on the $n$th anti-diagonal that
\[
  \sum_{n = r+s} \BS_{r,s} = 0.
\]
\end{lemma}

\begin{proof}
For odd $n \geq 1$ the sum is cancelled out by \eqref{eq:bs-recip}.
Now let $n \geq 2$ be even. Since we have $\BS_{n+1,0} = \BS_{0,n+1} = 0$
by \eqref{eq:bs-bn}, we derive from \eqref{eq:bs-recur} that
\[
  \sum_{n = r+s} \BS_{r,s} = \sum_{n = r+s} (\BS_{r+1,s} - \BS_{r,s+1})
    = \sum_{n+1 = r'+s'} (\BS_{r',s'} - \BS_{r',s'}) = 0
\]
by extending and changing the summation. This completes the proof.
\end{proof}

As a preliminary result, there exist on each anti-diagonal for $n \geq 1$ at
least two nonzero numbers with opposite sign, since the numbers $\BS_{1,s}$
and $\BS_{s,1}$ never vanish, whereas the sum on the anti-diagonal does.

It is not obvious that the numbers $\BS_{r,s}$ for $r,s \geq 2$ do not vanish
in general, since they may differ in sign and obey recurrences by means of
iterated sums and differences, as shown above.

For this reason, after introducing some definitions, we present the following
two main results. Henceforth, let $p$ denote a prime. Define the denominators
\begin{alignat*}{3}
  \DC_{r,s} &= \denom(\BS_{r,s}) &\quad& (r,s \geq 0), \\
  \DN_n &= \denom(\BN_n) && (n \geq 0).
\end{alignat*}
Let $\res{x}_m$ denote the least positive residue of $x \pmod{m}$.

\begin{theorem} \label{thm:main-1}
For $r,s \geq 0$, the numbers $\BS_{r,s}$ never vanish, except for the
cases $(r,s) = (n,0)$ and $(r,s) = (0,n)$ when $n \geq 3$ is odd.
\end{theorem}

\begin{theorem} \label{thm:main-2}
If $r,s \geq 1$, then we have
\[
  \DC_{r,s} = 2^{\varepsilon_2} \times 3 \times \mspace{-10mu}
    \prod_{\substack{5 \,\leq\, p \,\leq\, r+s+1 \\
      \res{r}_{p-1}+\res{s}_{p-1} \,\geq\, p-1}}
    \mspace{-10mu} p,
\]
where $\varepsilon_2 = 1$ if $r=1$ or $s=1$ with $r \neq s$,
otherwise $\varepsilon_2=0$. In the remaining cases, we have for $n \geq 0$ that
\[
  \DC_{n,0} = \DC_{0,n} = \DN_n.
\]
In particular, $\DC_{r,s}$ is squarefree for all $r,s \geq 0$ and obeys
the reciprocity relation
\[
  \DC_{r,s} = \DC_{s,r}.
\]
\end{theorem}

In a recent paper~\cite{Kellner:2021} of the author, self-reciprocal Bernoulli
polynomials, whose coefficients $\BF_{n,k}$ are related to the numbers $\BS_{r,s}$,
appeared in the context of power sums and the so-called Faulhaber polynomials,
while the reciprocity relation~\eqref{eq:bs-recip} casually followed as a corollary.
Therefore, the properties of $\BS_{r,s}$ are of certain interest to show the
non-vanishing of the aforementioned coefficients $\BF_{n,k}$, besides exceptional
cases.

Denneberg and Grabisch~\cite{Denneberg&Grabisch:1999}, for example,
considered similar numbers $b_m^d$. Due to different definition and indexing,
there is the relationship $b_m^d = \BS_{m,d-m}$, whereas Lemma~\ref{lem:diag-sum}
and the reciprocity relation~\eqref{eq:bs-recip} are formulated similarly.

The paper is organized as follows. The next section is devoted to preliminaries.
In Section~\ref{sec:umbral}, we use umbral calculus to give a short proof of
the reciprocity relation~\eqref{eq:bs-recip}, which can be extended to
Bernoulli polynomials in~\eqref{eq:bp-recip} as well.
Section~\ref{sec:prop} shows further properties of $\BS_{r,s}$ and $\DC_{r,s}$.
The last section contains the proofs of the main theorems above.


\section{Preliminaries}
\label{sec:prel}

For usual notation and definitions, we refer to the standard books
\cite{Comtet:1974,Graham&others:1994}. The forward difference operator and
its powers~$\Delta^n$ for $n \geq 0$ are given by
\[
  \Delta^{\!n} f(x) = \sum_{\nu=0}^n \binom{n}{\nu} (-1)^{n-\nu} f(x+\nu),
\]
where we use the expression, for example, $\Delta^n f(x+y) \valueat{x = 1}$
to indicate the variable and an initial value when needed.

The usual definition of the denominator $\denom(q)$ for $q \in \QQ$
is the smallest positive integer $d$ such that $d \cdot q \in \ZZ$.
In particular, $\denom(q) = 1$ if and only if $q \in \ZZ$.

By the von~Staudt--Clausen theorem, independently found by
von Staudt~\cite{Staudt:1840} and Clausen~\cite{Clausen:1840} in 1840,
the Bernoulli numbers with even indices satisfy
\begin{equation} \label{eq:bn-frac}
  \BN_n + \sum_{p-1 \mids n} \frac{1}{p} \in \ZZ \quad (n \in 2\NN).
\end{equation}
As a consequence, it follows for all $n \geq 0$ (cf.~Table~\ref{tbl:bs-coeff}) that
\begin{equation} \label{eq:bn-denom}
  \DN_n = \begin{cases}
    1, & \text{if $n = 0$}, \\
    2, & \text{if $n = 1$}, \\
    1, & \text{if $n \geq 3$ is odd}, \\
    \prod\limits_{p-1 \mids n} p, & \text{if $n \geq 2$ is even}.
  \end{cases}
\end{equation}

The Bernoulli polynomials are defined for $n \geq 0$ by
\begin{align}
  \BN_n(x) &= \sum_{\nu=0}^{n} \binom{n}{\nu} \BN_{n-\nu} \, x^\nu, \label{eq:bp-def} \\
\shortintertext{which obey the translation formula}
  \BN_n(x+y) &= \sum_{\nu=0}^{n} \binom{n}{\nu} \BN_{n-\nu}(x) \, y^\nu \label{eq:bp-trans} \\
\shortintertext{as well as the reflection relation}
  \BN_n(1-x) &= (-1)^n \, \BN_n(x). \label{eq:bp-refl}
\end{align}
It follows that $\BN_n(0) = \BN_n$ and $\BN_n(1) = (-1)^n \, \BN_n$.

We further need the following congruence. Hermite~\cite{Hermite:1876} showed
the case for odd $m$, while Stern~\cite{Stern:1878} gave a proof for the
general case, see also Bachmann~\cite[Eq.~(116), p.~46]{Bachmann:1910}.

\begin{lemma} \label{lem:bin-congr}
If $p$ is a prime and $m \geq 1$, then
\begin{equation} \label{eq:bin-congr}
  \sum_{\substack{\nu = 1\\p-1 \mids \nu}}^{m-1}
    \binom{m}{\nu} \equiv 0 \pmod{p}.
\end{equation}
\end{lemma}


\section{Umbral calculus}
\label{sec:umbral}

Using the umbral notation $\BN^n = \BN_n$, the relation
$\BN_n(1) = (-1)^n \, \BN_n$ becomes
\[
  (\BN+1)^n = (-\BN)^n.
\]
By linearity, it follows for any polynomial $f$ that
\[
  f(\BN+1) = f(-\BN).
\]
Taking $f(y) = y^r (y-1)^s$ for $r,s \geq 0$, this easily yields
\[
  (\BN+1)^r \, \BN^s = (-1)^{r+s} \, \BN^r (\BN+1)^s,
\]
which equals the reciprocity relation \eqref{eq:bs-recip}. This classical approach,
for example, is shown by Gessel~\cite[Lem.\,7.2, p.\,416]{Gessel:2003}.
In 1922, Tits~\cite[p.\,191]{Tits:1922} similarly noticed that
\[
  (\BN+1)^r \, \BN^s = \BN^r (\BN-1)^s
\]
(used with an old definition of $\BN_n$ in \cite{Tits:1922}),
which he attributed to Stern~\cite{Stern:1878b} of 1878.
See also Lucas~\cite[Sec.\,135, p.\,240]{Lucas:1891} of 1891, who also applied
this umbral trick. However, the reciprocity relation actually goes back to
von~Ettingshausen~\cite[Lec.\,42, Eq.\,(65), p.\,284]{Ettingshausen:1827} in 1827,
who used finite differences to achieve the result.

In particular, such formulas are of certain interest, since they lead to
shortened recurrences of the Bernoulli numbers. For example, one obtains
formulas involving about half of the numbers, i.e., $\BN_n, \ldots, \BN_{2n}$,
see \cite[Lec.\,42, Eqs.\,(66),(67), pp.\,284--285]{Ettingshausen:1827}
and \cite{Tits:1922}.

By exploiting these methods, Lehmer~\cite{Lehmer:1935} even derived a
recurrence using the numbers $\BN_{2n-12\lambda}$ with
$0 \leq \lambda \leq \floor{n/6}$, so effectively only $n/6$ nonzero terms are
needed. However, Carlitz~\cite{Carlitz:1964} showed that such recurrences cannot
be shortened arbitrarily, and so far all known recurrences contain $O(n)$
Bernoulli numbers. See also Agoh and Dilcher \cite{Agoh&Dilcher:2008}.

Now, we switch from Bernoulli numbers to polynomials.
If we extend definition~\eqref{eq:bs-def} to
\[
  \BS_{r,s}(x) = \sum_{\nu=0}^{r} \binom{r}{\nu} \BN_{s+\nu}(x),
\]
then the formulas mainly transfer from $\BS_{r,s}$ to $\BS_{r,s}(x)$,
where $\BS_{r,s} = \BS_{r,s}(0)$. For example,
\[
  \BS_{r+1,s}(x) = \BS_{r,s}(x) + \BS_{r,s+1}(x).
\]
In contrast, the generalized reciprocity relation holds asymmetrically by
\begin{equation} \label{eq:bp-recip}
  (-1)^r \BS_{r,s}(x) = (-1)^s \BS_{s,r}(-x).
\end{equation}

This can be shown as follows. From~\eqref{eq:bp-def} -- \eqref{eq:bp-refl},
we infer that
\[
  \sum_{\nu=0}^{n} \binom{n}{\nu} \BN_\nu(x) = \BN_n(x+1) = (-1)^n \, \BN_n(-x).
\]
Hence,
\[
  (\BN(x)+1)^n = (-\BN(-x))^n.
\]
Again, applying the polynomial $f(y) = y^r (y-1)^s$ as above, we finally derive that
\[
  (\BN(x)+1)^r \, \BN(x)^s = (-1)^{r+s} \, \BN(-x)^r \, (\BN(-x)+1)^s.
\]


\section{Further properties}
\label{sec:prop}

In this section, we discuss further results on the denominator $\DC_{r,s}$
that complement Theorem~\ref{thm:main-2}. They are also needed for the proofs
of the main theorems.

\begin{theorem} \label{thm:bs-denom}
Let $r,s \geq 0$. There are the following properties:
\begin{enumerate}
\item $\DC_{r,s} = \DC_{s,r}$.
\item $\DC_{0,s} = \DN_s$.
\item $\DC_{1,s} = \begin{cases}
         2, & \text{if $s = 0$}, \\
         3, & \text{if $s = 1$}, \\
         \DN_s, & \text{if $s \geq 2$ is even}, \\
         \DN_{s+1}, & \text{if $s \geq 3$ is odd}.
       \end{cases}$
\item $2 \nmid \DC_{r,s}$ for $r, s \geq 2$.
\item $3 \mid \DC_{r,s}$ for $r, s \geq 1$.
\item $p \mid \DC_{r,s}$ for even $r \geq 2$, where $p \geq 3$ and $p-1 \mid r$.
\end{enumerate}
In particular, $\DC_{r,s}$ is squarefree, and $\DC_{r,s} = 1$ only for the
cases $(r,s) = (0,0)$, $(r,s) = (n,0)$, and $(r,s) = (0,n)$
when $n \geq 3$ is odd.
\end{theorem}

The proof of Theorem~\ref{thm:bs-denom} is postponed to the end of this section,
since it relies on further results that we shall consider below.
We need the following definitions.

Let $\ZZ_p$ be the ring of $p$-adic integers and $\QQ_p$ be the field of
$p$-adic numbers.
Define for $r,s \geq 0$ and any prime $p$ the sum
\begin{align}
  \Psi_{r,s}(p) &=
    \sum_{\substack{\nu=0\\2 \mids s+\nu\\p-1 \mids s+\nu}}^{r}
    \mspace{-5.4mu} \binom{r}{\nu},
  \label{eq:psi-def} \\
\shortintertext{where we have that}
  \Psi_{r,s}(p) &= 0 \quad (p > s+r+1).
  \label{eq:psi-0}
\end{align}
Thus, summing $\Psi_{r,s}(p)$ over the primes satisfies that
\begin{equation} \label{eq:psi-sum}
  \sum_{p} \Psi_{r,s}(p) < \infty.
\end{equation}
A congruence of the type $\Psi_{r,s}(p) \equiv a \pmod{p}$ can be seen as a
generalization of congruence~\eqref{eq:bin-congr}.

Using properties of $\Psi_{r,s}(p)$,
we arrive at a von~Staudt--Clausen relation for $\BS_{r,s}$ and $\DC_{r,s}$,
similar to the classical relations \eqref{eq:bn-frac} and \eqref{eq:bn-denom}
for $\BN_n$ and $\DN_n$, respectively.

\begin{theorem} \label{thm:bs-denom-2}
Let $r,s \geq 2$. For any prime $p$, we have the reciprocity relation
\begin{equation} \label{eq:psi-recip}
  (-1)^r \, \Psi_{r,s}(p) \equiv (-1)^s \, \Psi_{s,r}(p) \pmod{p}.
\end{equation}
We have the integrality relation
\[
  \BS_{r,s} + \sum_{p} \frac{1}{p} \, \Psi_{r,s}(p) \in \ZZ,
\]
which implies the denominator formula
\begin{equation} \label{eq:dc-prod}
  \DC_{r,s} = \prod_{p \nmids \Psi_{r,s}(p)} p,
\end{equation}
where the index in the sum, respectively product,
can be bounded by $3 \leq p \leq r+s+1$.

\noindent Moreover, there are the following divisibility properties:
\begin{enumerate}
\item $2 \mid \Psi_{r,s}(2)$.
\item $3 \nmid \Psi_{r,s}(3)$.
\item $p \nmid \Psi_{r,s}(p)$, if $p \geq 5$ where $p-1 \mid r$ or $p-1 \mid s$.
\end{enumerate}
\end{theorem}

\begin{proof}
By the von~Staudt--Clausen theorem \eqref{eq:bn-frac}, we obtain for even
$n \geq 2$ that
\[
   - \BN_n \equiv \sum_{p-1 \mids n} \frac{1}{p} \pmod{\ZZ}.
\]
Let $r,s \geq 2$. Then equation~\eqref{eq:bs-def} turns into
\[
  - \BS_{r,s} \equiv \sum_{\nu=0}^{r} \binom{r}{\nu} \mspace{-8mu}
    \sum_{\substack{2 \mids s+\nu\\p-1 \mids s+\nu}} \mspace{-3mu}
    \frac{1}{p} \pmod{\ZZ},
\]
since we only have to take the even indexed Bernoulli numbers into account.
By changing the summation over the primes and considering
\eqref{eq:psi-0} and \eqref{eq:psi-sum}, we finally infer that
\begin{equation} \label{eq:bs-frac-p}
  - \BS_{r,s} \equiv \sum_{p} \frac{1}{p} \, \Psi_{r,s}(p) \pmod{\ZZ},
\end{equation}
where almost all sums $\Psi_{r,s}(p)$ are empty. This yields
\begin{equation} \label{eq:bs-int}
  \BS_{r,s} + \sum_{p} \frac{1}{p} \, \Psi_{r,s}(p) \in \ZZ.
\end{equation}

Only those primes $p$ give a contribution to the denominator of $\BS_{r,s}$
for which $p \nmids \Psi_{r,s}(p)$. Thus, it follows that
\begin{equation} \label{eq:dc-psi}
  \DC_{r,s} = \prod_{p \nmids \Psi_{r,s}(p)} p.
\end{equation}

Furthermore, by \eqref{eq:bs-recip} we have
\begin{align}
  (-1)^r \BS_{r,s} &\equiv (-1)^s \BS_{s,r} \pmod{\ZZ},
  \nonumber \\
\shortintertext{so by congruence \eqref{eq:bs-frac-p} this turns into}
  (-1)^r \sum_{p} \frac{1}{p} \, \Psi_{r,s}(p) &\equiv
    (-1)^s \sum_{p} \frac{1}{p} \, \Psi_{s,r}(p) \pmod{\ZZ}.
  \label{eq:psi-frac} \\
\shortintertext{Lifting to $\pmods{p}$ for any prime $p$ then provides that}
  (-1)^r \, \Psi_{r,s}(p) &\equiv (-1)^s \, \Psi_{s,r}(p) \pmod{p}.
  \label{eq:psi-sym}
\end{align}
Actually, we transfer \eqref{eq:psi-frac} lying in $\QQ_p$ to \eqref{eq:psi-sym}
lying in $\ZZ_p$. This is allowed here due to the occurring factor
$\frac{1}{p}$, which is then cancelled out.

Now we show the remaining three parts:

(i),~(ii). By definition of~\eqref{eq:psi-def}, we have
$\Psi_{r,s}(2) = \Psi_{r,s}(3)$. Depending on the parity of $s$, we infer that
\[
  \Psi_{r,s}(2) = \sum_{\substack{\nu=0\\2 \mids s+\nu}}^{r}
    \mspace{-5mu} \binom{r}{\nu}
    = \frac{1}{2} \left( (1+1)^r \pm (1-1)^r \right) = 2^{r-1},
\]
where \phrase{$\pm$} is replaced by \phrase{$+$} if $2 \mid s$,
otherwise by \phrase{$-$} if $2 \nmid s$. Since $r \geq 2$,
it follows that $2 \mid \Psi_{r,s}(2)$ and $3 \nmid \Psi_{r,s}(3)$.

(iii). Let $p \geq 5$. Due to symmetry by~\eqref{eq:psi-sym},
it suffices to consider the condition $p-1 \mid s$.
Using congruence~\eqref{eq:bin-congr} of Lemma~\ref{lem:bin-congr},
we deduce that
\[
  \Psi_{r,s}(p) \equiv
    \sum_{\substack{\nu \,\in\, \set{0,r}\\p-1 \mids s+\nu}}
    \mspace{-5mu} \binom{r}{\nu}
    \equiv a \pmod{p}
\]
with $a \in \set{1,2}$. Thus, $p \nmid \Psi_{r,s}(p)$.

Finally, by parts (i) and (ii), and \eqref{eq:psi-0},
the index $p$ in the sum of~\eqref{eq:bs-int}, respectively product
of~\eqref{eq:dc-psi}, is bounded by $3 \leq p \leq r+s+1$.
This completes the proof of the theorem.
\end{proof}

Moreover, there are some refinements of the properties of $\Psi_{r,s}(p)$.

\begin{lemma} \label{lem:psi-congr-1}
Let $r,r' \geq 1$, $s,s' \geq 0$, and $p \geq 3$ be a prime.
There are the following properties:
\begin{enumerate}
\item If $s \equiv s' \pmod{p-1}$, then
      $\Psi_{r,s}(p) = \Psi_{r,s'}(p)$.
\item If $r \equiv r' \pmod{p-1}$, then
      $\Psi_{r,s}(p) \equiv \Psi_{r',s}(p) \pmod{p}$.
\item The reciprocity relation \eqref{eq:psi-recip} extends to $r,s \geq 1$.
\end{enumerate}
\end{lemma}

\begin{proof}
Let $p \geq 3$. We have to show three parts:

(i).~The condition in the sum~\eqref{eq:psi-def} remains unchanged by
$s + \nu \equiv s' +\nu \pmod{p-1}$, so $\Psi_{r,s}(p) = \Psi_{r,s'}(p)$.

(ii).~Note that
\[
  \Psi_{p,s}(p) \equiv
    \sum_{\substack{\nu = 0\\p-1 \mids s+\nu}}^{p} \mspace{-9mu} \binom{p}{\nu}
    \equiv \sum_{\substack{\nu \,\in\, \set{0,p}\\p-1 \mids s+\nu}} 1
    \equiv \sum_{\substack{\nu \,\in\, \set{0,1}\\p-1 \mids s+\nu}} 1
    \equiv \Psi_{1,s}(p) \pmod{p}.
\]
If $r=1$ or $r'=1$, then by the above congruence we switch to $r=p$ or $r'=p$,
respectively. If $0 \leq s < 2$, then by part~(i) we switch to $s \geq 2$.
Now, we can assume that $r,r',s \geq 2$.
Using part (i) and applying Theorem~\ref{thm:bs-denom-2} yield
\[
  \Psi_{r,s}(p) \equiv (-1)^{r+s} \, \Psi_{s,r}(p) \equiv
    (-1)^{r'+s} \, \Psi_{s,r'}(p) \equiv \Psi_{r',s}(p) \pmod{p},
\]
noting that $r$ and $r'$ have the same parity, since $p-1$ is even.

(iii).~Assume that $r,s \geq 1$. Choose $r',s' \geq 2$
with $s \equiv s' \pmod{p-1}$ and $r \equiv r' \pmod{p-1}$.
Using parts (i) and (ii) as well as applying Theorem~\ref{thm:bs-denom-2},
we infer that
\[
  (-1)^r \Psi_{r,s}(p) \equiv (-1)^{r'} \Psi_{r',s'}(p)
    \equiv (-1)^{s'} \Psi_{s',r'}(p) \equiv (-1)^s \Psi_{s,r}(p) \pmod{p},
\]
where $r$ and $r'$, respectively $s$ and $s'$, have the same parity.
\end{proof}

\begin{lemma} \label{lem:psi-congr-2}
Let $p \geq 5$ be a prime. Then we have the following congruence between matrices:
\[
  \Big(\Psi_{r,s}(p)\Big)_{1 \leq r,s \leq p-2} \equiv
    \left( \begin{array}{ccccc}
        &   &   &   & 1 \\
        &   &   & 1 & \star \\
        &   & \iddots & \vdots & \vdots \\
        & 1 & \cdots & \star & \star \\
      1 & \star & \cdots & \star & \star \\
    \end{array} \right) \pmod{p},
\]
where all empty entries are zero, the entries on the anti-diagonal are one,
and the entries below are nonzero $\pmods{p}$. More precisely, we have for
$1 \leq r,s \leq p-2$ the properties:
\begin{enumerate}
\item $\Psi_{r,s}(p) = 0$, if $r+s < p-1$.
\item $\Psi_{r,s}(p) = 1$, if $r+s = p-1$.
\item $p \nmid \Psi_{r,s}(p) = \binom{r}{p-1-s}$, if $r+s \geq p-1$.
\end{enumerate}
\end{lemma}

\begin{proof}
Let $p \geq 5$ and $1 \leq r,s \leq p-2$. Regarding the condition of the
sum~\eqref{eq:psi-def} of $\Psi_{r,s}(p)$, define
$\IM_{r,s} = \set{s,s+1,\dots,s+r} \cap (p-1)\NN$ as the set of numbers
$s +\nu$ divisible by $p-1$. We have to show three parts:

(i).~Since $r+s < p-1$, we have $\IM_{r,s} = \emptyset$ and $\Psi_{r,s}(p) = 0$.

(ii).~From $r+s = p-1$, it follows that $\IM_{r,s} = \set{r+s}$ and $\Psi_{r,s}(p) = 1$.

(iii).~We have $r+s \geq p-1$ and $r,s \leq p-2$.
Thus, $\IM_{r,s} = \set{p-1}$ and $\Psi_{r,s}(p) = \binom{r}{p-1-s}$.
Since $r < p$, this finally shows that $p \nmid \Psi_{r,s}(p)$.
\end{proof}

Now, we give the postponed proof of Theorem~\ref{thm:bs-denom}.

\begin{proof}[Proof of Theorem~\ref{thm:bs-denom}]
Let $r,s \geq 0$. We have to show six parts:

(i). This follows from~\eqref{eq:bs-recip}.

(ii),~(iii). Both relations for $\DC_{0,s}$ and $\DC_{1,s}$ follow
from~\eqref{eq:bs-bn} and Table~\ref{tbl:bs-coeff}.

(iv),~(v). Let $r,s \geq 2$. By Theorem~\ref{thm:bs-denom-2}~(i) and~(ii),
we have $2 \mid \Psi_{r,s}(2)$ and $3 \nmid \Psi_{r,s}(3)$,
implying that $2 \nmid \DC_{r,s}$ and $3 \mid \DC_{r,s}$, respectively.
Using parts (i) and (iii), it even follows that $3 \mid \DC_{r,s}$ holds
for $r,s \geq 1$.

(vi). Let $r \geq 2$ be even. Since $\DC_{r,0} = \DC_{r,1} = \DN_r$ by
parts~(i) -- (iii), the result holds for $s=0,1$. Now, let $s \geq 2$.
By Theorem~\ref{thm:bs-denom-2}~(ii) and~(iii), we have for $p \geq 3$
and $p-1 \mid r$ that $p \nmid \Psi_{r,s}(p)$, showing that $p \mid \DC_{r,s}$.

Since $\DN_n$ is squarefree by \eqref{eq:bn-denom}, $\DC_{r,s}$ must be
squarefree by~\eqref{eq:bs-def}, too. As a result, we infer from part~(v) that
$\DC_{r,s} > 1$ for $r,s \geq 1$. Parts~(i) and~(ii) imply for the remaining
cases that $\DC_{0,n} = \DC_{n,0} = \DN_n$ for $n \geq 0$.
By \eqref{eq:bn-denom} we finally identify the only cases
where $\DC_{r,s} = 1$, namely,
$(r,s) = (0,0)$, $(r,s) = (n,0)$, and $(r,s) = (0,n)$ for odd $n \geq 3$.
This proves the theorem.
\end{proof}


\section{Proofs of main theorems}
\label{sec:proofs}

\begin{proof}[Proof of Theorem~\ref{thm:main-1}]
Let $r,s \geq 0$. Note that $\DC_{r,s} = 1 \iffs \BS_{r,s} \in \ZZ$.
By Theorem~\ref{thm:bs-denom} we have the cases $(r,s) = (0,0)$, $(r,s) = (n,0)$,
and $(r,s) = (0,n)$ when $n \geq 3$ is odd.
Since $\BS_{0,0} = 1$ by Table~\ref{tbl:bs-coeff}, there remain the cases
$\BS_{n,0} = \BS_{0,n} = \BN_n = 0$ by \eqref{eq:bs-bn}, showing the result.
\end{proof}

\begin{proof}[Proof of Theorem~\ref{thm:main-2}]
We have to show for $r,s \geq 1$ that
\begin{equation} \label{eq:dc-prod-2}
  \DC_{r,s} = 2^{\varepsilon_2} \times 3 \times \mspace{-10mu}
    \prod_{\substack{5 \,\leq\, p \,\leq\, r+s+1 \\
      \res{r}_{p-1}+\res{s}_{p-1} \,\geq\, p-1}}
    \mspace{-10mu} p
\end{equation}
with the exponent $\varepsilon_2$ as defined before. We first consider the
factors $2$ and $3$ of $\DC_{r,s}$. By symmetry of $\DC_{r,s}$, the definition
of $\varepsilon_2$ follows from Theorem~\ref{thm:bs-denom} (iii) and (iv),
while the factor~$3$ follows from Theorem~\ref{thm:bs-denom} (v).
Next, we consider the prime factors $p \geq 5$ in the case $r=1$,
which implies by symmetry the case $s=1$ as well.
If $p \geq 5$ and $p \mid \DC_{1,s}$, then by Theorem~\ref{thm:bs-denom} (iii)
we have $p \mid \DN_{s+\delta}$ for $s \geq 2$, where $\delta = 0$ for
even $s \geq 2$, and $\delta = 1$ for odd $s \geq 3$.
Using \eqref{eq:bn-denom}, it then follows that $p-1 \mid s + \delta$.
Consequently, we obtain
\[
  p \leq s + \delta + 1 \leq r+s+1 \andq
  \res{r}_{p-1}+\res{s}_{p-1} = 1 + p - 1 - \delta \geq p-1,
\]
which are compatible with the bound and condition
of product~\eqref{eq:dc-prod-2}, respectively.

As a result, we can now assume that $r,s \geq 2$ and $p \geq 5$.
Using Theorem~\ref{thm:bs-denom-2}, we shall show that the product
formulas~\eqref{eq:dc-prod} and~\eqref{eq:dc-prod-2} are equivalent.
Therefore, it remains to show that
\begin{equation} \label{eq:psi-res}
  p \nmid \Psi_{r,s}(p) \iffq \res{r}_{p-1} + \res{s}_{p-1} \geq p-1.
\end{equation}

The case $p-1 \mid r$ (and so $p-1 \mid s$ by symmetry) follows from
Theorem~\ref{thm:bs-denom-2} (iii), since $\res{r}_{p-1} = p-1$.
For the remaining cases, we can switch by Lemma~\ref{lem:psi-congr-1}
and the definition of $\res{\cdot}_{p-1}$ to $r' = \res{r}_{p-1}$ and
$s' = \res{s}_{p-1}$, both variables $r', s'$ being in the
residue system $\set{1,\dots,p-2} \pmod{p-1}$.
Finally, we infer from Lemma~\ref{lem:psi-congr-2} that
\[
  p \nmid \Psi_{r',s'}(p) \iffq r' + s' \geq p-1,
\]
showing \eqref{eq:psi-res} completely. The remaining parts are given by
Theorem~\ref{thm:bs-denom} (i) and (ii). This completes the proof of
the theorem.
\end{proof}

\newpage


\appendix
\section{}

\vspace*{-2ex}
\begin{figure}[H]
\begin{center}
\includegraphics[width=12cm]{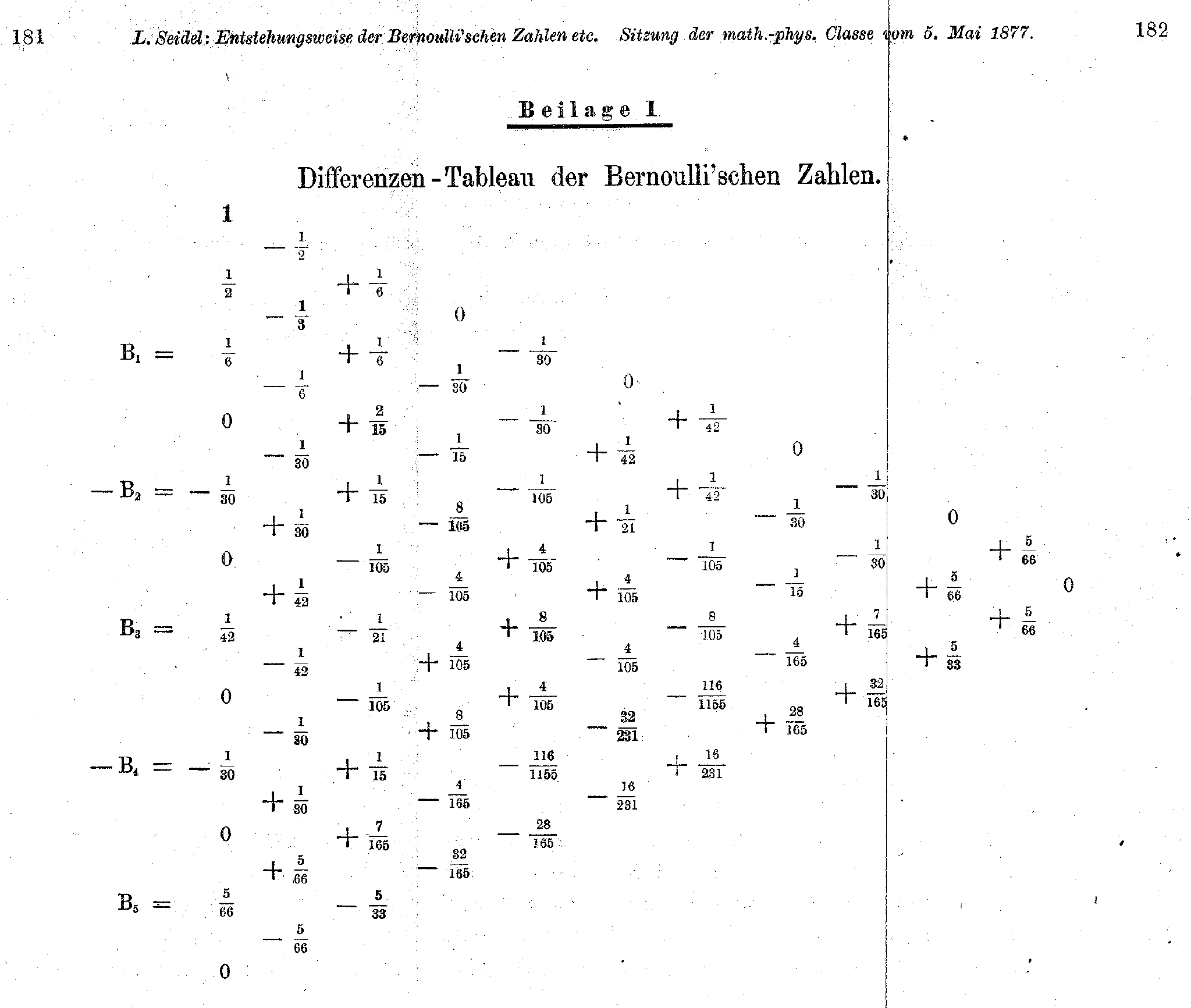}

\caption{\parbox[t]{23em}{Seidel's table of the differences of $\BN_n$ (numbers in old notation).}}
\label{fig:seidel}
\end{center}
\end{figure}

The reprinted table from Seidel~\cite{Seidel:1877} is copyrighted by
BAdW~\cite{BAdW} under the Creative Commons license CC-BY.


\section*{Acknowledgements}

We would like to thank the anonymous referee for several suggestions.


\bibliographystyle{amsplain}

\end{document}